\documentclass[oneside,a4paper]{article}

\usepackage{setspace}
\usepackage{amsthm}
\usepackage{amsmath,amssymb}
\usepackage[T1]{fontenc}
\usepackage[polish,english]{babel}
\usepackage[utf8]{inputenc}
\usepackage{url}
\usepackage{calc}
\usepackage{graphicx}
\usepackage{fancyhdr}
\usepackage{color}
\usepackage{natbib}
\usepackage{multirow}
\usepackage{comment}
\usepackage{array}
\usepackage{tabulary,tabularx}
\usepackage{threeparttable}
\usepackage{lscape}
\usepackage{ragged2e}

\newcolumntype{P}[1]{>{\RaggedRight\hspace{0pt}}p{#1}}
\newcolumntype{L}[1]{>{\raggedright\let\newline\\\arraybackslash\hspace{0pt}}m{#1}}
\newcolumntype{C}[1]{>{\centering\let\newline\\\arraybackslash\hspace{0pt}}m{#1}}
\newcolumntype{R}[1]{>{\raggedleft\let\newline\\\arraybackslash\hspace{0pt}}m{#1}}

\newtheorem{theorem}{Theorem}
\newtheorem{lemma}{Lemma}
\newtheorem{corollary}{Corollary}
\newcommand{\ud}{\mathrm{d}}

\begin{document}

\title{\large \textbf{A note on some extensions of the matrix angular central Gaussian distribution.}}

\date{}

\author{Justyna Wróblewska}

\author{Justyna Wróblewska\thanks{Cracow University of Economics, Department of Econometrics and Operational Research, Rakowicka 27, 31-510 Kraków, Poland, e-mail: eowroble@cyf-kr.edu.pl}}

\maketitle
\thispagestyle{fancy}
\fancyhead{}

\begin{abstract}
This paper extends the notion of the matrix angular central distribution (MACG) to the complex case. We start by considering the normally distributed random complex matrix ($Z$) and show that is the orientation ($H_Z=Z(Z'Z)^{-1}$) has complex MACG (CMACG) distribution. Then we discuss the distribution of the orientation of the linear transformation of the random matrix which orientation part has CMACG distribution. Finally, we discuss the family of distributions which lead to the CMACG distribution.
\end{abstract}

\begin{center} \end{center}



\newpage

\section{Introduction}
\label{sec:intro}

The density and properties of the matrix angular central distribution (MACG) were introduced by \citet{Chikuse1990}. The distribution is defined for the elements of the Stiefel manifold. MACG distribution proofed to be very useful in the Bayesian analysis of cointegration \citet{Koop_al2009} and in Bayesian models combining cointegration with the idea of common cyclical features (see \citealp{Wroblewska2011}, \citeyear{Wroblewska2012}, \citeyear{Wroblewska2015}). An easy way of obtaining the pseudo-random sample from the MACG distribution belongs to its main advantages in the Bayesian analyses. Moreover, as MACG distribution is invariant to the right orthonormal transformations, it can be treated as the distributions defined on the Grassmann manifolds. This feature is an advantage in the above-mentioned analysis, the data contain information only about cointegration and common feature spaces, not about the vectors spanning them.  Finally, through the parameter of MACG distribution, the researcher can easily and transparently incorporate prior information about the analysed spaces. However, if the researcher is interested in the analysis of the seasonally cointegrated process (see e.g. \citealp{Hylleberg_al1990}, \citealp{Johansen_Schaumburg1999}, \citealp{Cubadda_Omtzigt2005}) within the Bayesian paradigm, the generalization of MACG to the complex case may be useful (see \citealp{Wroblewska2020}).

We start with the basic definitions and measure decomposition and then move to the definition of the complex matrix angular central distribution (CMACG) and its properties.

The set of $m\times r$ $(m\geq r)$ semi-unitary matrices, i.e. matrices fulfilling the condition $\bar{X}'X=I_r$, where $\bar{X}'$ denotes the conjugate transpose of $X$ and $I_r$ is the $r\times r$ identity matrix, is called the complex Stiefel manifold ($V_{r,m}^{\mathbb{C}}$):
$$V_{r,m}^{\mathbb{C}}=\left\{X_{m\times r}: \bar{X}'X=I_r,\ m\geq r\right\}.$$

An invariant measure on $V_{r,m}^{\mathbb{C}}$ is given by the differential form \citep{Di_Gu2011}:
$$(\bar{X}'\ \ud X)=\bigwedge_{i=1}^m\bigwedge_{j=i+1}^r\bar{x}_j'\ \ud x_i,$$
where $\bigwedge$ denotes the exterior product and the matrix $X_1$ is chosen such that $\mathbf{X}=\left(X,\ X_1\right)$ is an element of the unitary group ($\bar{\mathbf{X}}'\mathbf{X}=I_m$).\\
The volume of the complex Stiefel manifold is
$$Vol(V_{r,m}^{\mathbb{C}})=\int_{X\in V_{r,m}^{\mathbb{C}}}(\bar{X}'\ \ud X)=\frac{2^r\pi^{mr}}{\Gamma_r^{\mathbb{C}}[m]},$$
where  $\Gamma_r^{\mathbb{C}}[a]$ denotes the complex multivariate Gamma function, and is defined by:
$$\Gamma_r^{\mathbb{C}}[a]=\int_{A_{r\times r}>0, \bar{A}'=A}\exp\{-tr(A)\}|A|^{a-r}(\ud A)=\pi^{r(r-1)/2}\prod_{i=1}^r\Gamma[a-i+1],$$
where $tr(\cdot)$ denotes the trace, $|\cdot|$ - the determinant and $Re(a)>m-1$ (see \citealp{Gross_Richards1987}, \citealp{Di_Gu2011}).\\
The normalized invariant measure ($[dX]$) of unit mass on the considered manifold is defined as:
\begin{equation}
[\ud X]=\frac{(\bar{X}'\ \ud X)}{Vol(V_{r,m}^{\mathbb{C}})}=\frac{\Gamma_r^{\mathbb{C}}[m]}{2^r\pi^{mr}}(\bar{X}'\ \ud X).
\label{eq:inv}
\end{equation}

The next two theorems provide to Jacobians of the transformation which will be used through the paper.

\begin{theorem}[\citealp{Di_Gu2011}]
If $Y=AXB+C$, where $X\in\mathbb{C}^{m\times r}$ and $Y\in\mathbb{C}^{m\times r}$ are random matrices and $A\in\mathbb{C}^{m\times m},\ |A|\neq0$, $B\in\mathbb{C}^{r\times r},\ |B|\neq0$, $C\in\mathbb{C}^{m\times r}$ are matrices of constants, then
\begin{equation}
(\ud Y)=|\bar{A}'A|^r|\bar{B}'B|^m(\ud X),
\end{equation}
so that $J(Y\rightarrow X)=|\bar{A}'A|^r|\bar{B}'B|^m$.
\end{theorem}

\begin{theorem}[\citealp{Polcari2017}]
If $Y=BX\bar{B}'$, where $X\in\mathbb{C}^{m\times m}$ and $Y\in\mathbb{C}^{m\times m}$ are random Hermitian matrices $(\bar{X}'=X,\ \bar{Y}'=Y)$ and $B\in\mathbb{C}^{m\times m}$ is a non-singular $(|B|\neq0)$ matrix of constants, then
\begin{equation}
(\ud Y)=|B|^{2m}(\ud X),
\label{eq:jak}
\end{equation}
so that $J(Y\rightarrow X)=|B|^{2m}$.
\end{theorem}

\section{Complex Matrix Angular Central Gaussian Distribution}
\label{sec:CMACG}

Following the idea of the MACG distribution of Chikuse (\citeyear{Chikuse1990}, \citeyear{Chikuse2003}) we analyze the distribution of the "orientation" part ($H_Z$) of polar decomposition of the full column rank random matrix $Z_{m\times r}$, $m\geq r$, $r(Z)=r$.\\
The unique polar decomposition of $Z$ is defined as:
$$Z=H_ZT_Z^{\frac{1}{2}},\quad H_Z=Z(\bar{Z}'Z)^{-\frac{1}{2}},\quad T_Z=\bar{Z}'Z.$$

\begin{lemma}
The measure $(\ud Z)$ is decomposed as
 \begin{equation}
(\ud Z)=\frac{\pi^{mr}}{\Gamma_r^{\mathbb{C}}[m]}|T_Z|^{m-r}(\ud T_Z)[\ud H_Z].
\label{eq:dek}
\end{equation}
\label{lem:dek}
\end{lemma}
\begin{proof}
It is the straightforward consequence of the decomposition of the measure\\ $(\ud Z)=2^{-r}|T_Z|^{m-r}(\ud T_Z)(\bar{H}'_Z\ \ud H_Z)$ (see \citealp{Di_Gu2011}) and the definition of the normalized invariant measure on the complex Stiefel manifold (see Equation \ref{eq:inv}).
\end{proof}
Using Lemma \ref{lem:dek} we obtain that the density of the orientation $H_Z$:
\begin{equation}
f_{H_Z}(H_Z)=\frac{\pi^{mr}}{\Gamma_r^{\mathbb{C}}[m]}\int_{T>0,\bar{T}'=T}f_Z(H_ZT^\frac{1}{2})|T|^{m-r}(\ud T).
\label{eq:densHZ}
\end{equation}
\begin{theorem}
Assume that $Z_{m\times r}$ has the $m\times r$ the matrix-variate complex central normal distribution with the parameter $P$, $Z\sim mN^{\mathbb{C}}(0,I_r,P)$, where $P$ is an $m\times m$ positive definite matrix and define $H_Z=Z(\bar{Z}'Z)^{-\frac{1}{2}}\in V_{r,m}^{\mathbb{C}}$.\\
Then it is said that $H_Z$ has a complex matrix angular central Gaussian distribution with parameter $P$, denoted as $H_Z\sim CMACG(P)$, and its density is
\begin{equation} 
f_{H_Z}(H_Z)=|P|^{-r}|\bar{H}'_ZP^{-1}H_Z|^{-m}.
\label{eq:CMACG}
\end{equation}
\label{th:CMACGdef}
\end{theorem}
\begin{proof}
The density of $Z$ is
$$f_Z(Z)=\pi^{-mr}|P|^{-r}\exp[-tr(\bar{Z}'P^{-1}Z)],$$
so according to (\ref{eq:dek}) the density of $H_Z$ is obtained as
\begin{eqnarray*}
f_{H_Z}(H_Z)&\underbrace{=}_{(\ref{eq:densHZ})}&\frac{\pi^{mr}}{\Gamma_r^{\mathbb{C}}[m]}\int_{T>0,\bar{T}'=T}f_Z(H_ZT^\frac{1}{2})|T|^{m-r}(\ud T)=\\
&=&\frac{\pi^{mr}}{\Gamma_r^{\mathbb{C}}[m]}\int_{T>0,\bar{T}'=T}\pi^{-mr}|P|^{-r}\exp[-tr(T^{\frac{1}{2}}\bar{H}'_ZP^{-1}H_ZT^\frac{1}{2})]|T|^{m-r}(\ud T)=\\
&=&\frac{|P|^{-r}}{\Gamma_r^{\mathbb{C}}[m]}\int_{T>0,\bar{T}'=T}\exp[-tr(T^{\frac{1}{2}}\bar{H}'_ZP^{-1}H_ZT^\frac{1}{2})]|T|^{m-r}(\ud T).
\end{eqnarray*}
In the integral make the change of variables $V=M^{\frac{1}{2}}TM^{\frac{1}{2}}$, where $M$ stands for $\bar{H}'_ZP^{-1}H_Z$. By (\ref{eq:jak}) $(\ud T)=|M|^{-r}(\ud V)$ so the integral becomes
\begin{eqnarray*}
f_{H_Z}(H_Z)&=&\frac{|P|^{-r}}{\Gamma_r^{\mathbb{C}}[m]}\int_{V>0,\bar{V}'=V}\exp[-tr(V)]|VM^{-1}|^{m-r}|M|^{-r}(\ud V)=\\
&=&\frac{|P|^{-r}}{\Gamma_r^{\mathbb{C}}[m]}|M|^{-m}\int_{V>0,\bar{V}'=V}\exp[-tr(V)]|V|^{m-r}(\ud V)=\\
&=&\frac{|P|^{-r}}{\Gamma_r^{\mathbb{C}}[m]}|M|^{-m}\Gamma_r^{\mathbb{C}}[m]=\\
&=&|P|^{-r}|\bar{H}'_ZP^{-1}H_Z|^{-m}.
\end{eqnarray*}
\end{proof}
Note the distribution in question inherits the properties from its real counterpart.\\
There is an indeterminacy in the matrix parameter $P$ by multiplication by a positive scalar (i.e. CMACG($P$)=CMACG($cP$), where $c>0$). For $P=I_m$ the orientation $H_Z$ is uniformly distributed over the complex Stiefel manifold. It should be also emphasized that CMACG distribution is invariant under right unitary transformations ($H_Z\rightarrow H_ZQ,\ Q\in O(r)$), so it can be treated as the distribution defined on the complex Grassmann manifold.

The decomposition (\ref{eq:dek}) leads to the feature stated below (see \citealp{Chikuse1990}, Theorem 2.3 for the more extended discussion of the characterization of such distribution in the real case).
\begin{theorem}
If the $m\times r$ complex random matrix $Z$ has the density of the form $g(\bar{Z}'Z)$ then its orientation $H_Z$ is uniformly distributed on $V_{r,m}^{\mathbb{C}}$.
\label{th:gZZ}
\end{theorem}

\begin{proof}
With the help of (\ref{eq:densHZ}) we obtain:
\begin{eqnarray*}
f_{H_Z}(H_Z)&=&\frac{\pi^{mr}}{\Gamma_r^{\mathbb{C}}[m]}\int_{T>0,\bar{T}'=T}f_Z(H_ZT^\frac{1}{2})|T|^{m-r}(\ud T)=\\
&=&\frac{\pi^{mr}}{\Gamma_r^{\mathbb{C}}[m]}\int_{T>0,\bar{T}'=T}g(T^\frac{1}{2}\bar{H}'_ZH_ZT^\frac{1}{2})|T|^{m-r}(\ud T)=\\
&\underbrace{=}_{\bar{H}'_ZH_Z=I_r}&\frac{\pi^{mr}}{\Gamma_r^{\mathbb{C}}[m]}\int_{T>0,\bar{T}'=T}g(T)|T|^{m-r}(\ud T)=\\
&=&const.
\end{eqnarray*}
\end{proof}

\begin{theorem}
Let Z be an $m\times r$ complex random matrix with the density $f_Z(Z)$ invariant under right unitary transformation $(Z\rightarrow ZQ,\ \bar{Q}'Q=I_r)$. Define a new $m\times r$ random matrix $Y=BZ$ with an $m\times m$ non-singular matrix $B$, $(|B|\neq 0)$. Consider polar decomposition of these matrices:
\begin{itemize}
\item $Z=H_ZT_Z^{1/2}$ with $H_Z=Z(Z'Z)^{-1/2}$ and $T_Z=Z'Z$,
\item $Y=H_YT_Y^{1/2}$ with $H_Y=Y(Y'Y)^{-1/2}$ and $T_Y=Y'Y$.
\end{itemize}
and let $f_{H_Z}(H_Z)$ be the density of $H_Z$ (see Theorem \ref{th:CMACGdef}). Then the density of $H_Y$, the orientation of the random matrix $Y$, is of the form:
\begin{equation}
f_{H_Y}(H_Y)=|\bar{B}'B|^{-r}|W'W|^{-m}f_{H_Z}(H_W),
\label{eq:HYdens}
\end{equation}
where $W=B^{-1}H_Y$ and $H_W$ is the orientation of $W$, i.e. $H_W=W(W'W)^{-1/2}$.
\label{th:BZ}
\end{theorem}

\begin{proof}
Knowing the density of Z and the Jacobian of transformation $Z\rightarrow BZ=Y,\ (\ud Y)=|\bar{B}'B|^{r}(\ud Z)$ we may obtain the density of $Y$:
\begin{equation}
f_Y(Y)=|\bar{B}'B|^{-r}f_Z(B^{-1}Y),
\end{equation}
which together with (\ref{eq:densHZ}) leads to the density of $H_Y$:
\begin{equation}
f_{H_Y}(H_Y)=\frac{\pi^{mr}}{\Gamma_r^{\mathbb{C}}}|\bar{B}'B|^{-r}\int_{T>0,\bar{T}'=T}f_Z(B^{-1}H_YT^\frac{1}{2})|T|^{m-r}(\ud T).
\label{eq:fHY}
\end{equation}
We follow \citet{Chikuse1990} and apply the idea of her transformation (3.4) to the complex case:
\begin{equation}
T=(\bar{W}'W)^{-1/2}S(\bar{W}'W)^{-1/2},\ \text{with}\ W=B^{-1}H_Y,
\label{eq:trans}
\end{equation}
the Jacobian od this transformation leads to the relationship between measures $(\ud T)=|(\bar{W}'W)^{-/2}|^{2r}(\ud S)=|\bar{W}'W|^{-r}(\ud S)$.\\
From the invariance property of the density of $Z$ we have
\begin{equation}
f_Z(WT^{1/2})=f_Z(H_WS^{1/2}).
\label{eq:invfZ}
\end{equation}

Now we can combine the above stated transformation and present the density of $H_Y$ as:
\begin{eqnarray*}
f_{H_Y}(H_Y)&=&\frac{\pi^{mr}}{\Gamma_r^{\mathbb{C}}}|\bar{B}'B|^{-r}\times\\
&\times&\int_{S>0,\bar{S}'=S}f_Z(H_WS^\frac{1}{2})|(\bar{W}'W)^{-1/2}S(\bar{W}'W)^{-1/2}|^{m-r}|\bar{W}'W|^{-r}(\ud S)=\\
&=&|\bar{B}'B|^{-r}|\bar{W}'W|^{-m}\frac{\pi^{mr}}{\Gamma_r^{\mathbb{C}}}\int_{S>0,\bar{S}'=S}f_Z(H_WS^\frac{1}{2})|S|^{m-r}(\ud S)=\\
&\underbrace{=}_{(\ref{eq:densHZ})}&|\bar{B}'B|^{-r}|\bar{W}'W|^{-m}f_{H_Z}(H_W).
\end{eqnarray*} 
\end{proof}

Theorem \ref{th:BZ} leads to the following feature of the CMACG distribution for the linear transformations of complex random matrices.
\begin{corollary}
If $H_Z$, the orientation of $Z$, has the CMACG($P$) distribution, then $H_Y$, the orientation of $Y=BZ$, has the CMACG($BP\bar{B}'$) distribution.
\label{cor:BZ}
\end{corollary}

\begin{proof}
As the orientation $H_Z$ has CMACG($P$) distribution its density is $f_{H_Z}(H_Z)=|P|^{-r}|\bar{H}'_ZP^{-1}H_Z|^{-m}$, see (\ref{eq:CMACG}). Using (\ref{eq:HYdens}) form Theorem \ref{th:BZ} we obtain
\begin{eqnarray*}
f_{H_Y}(H_Y)&=&|\bar{B}'B|^{-r}|W'W|^{-m}f_{H_Z}(H_W)=\\
&=&|\bar{B}'B|^{-r}|W'W|^{-m}|P|^{-r}|\bar{H}'_WP^{-1}H_W|^{-m}=\\
&=&|\bar{B}'B|^{-r}|W'W|^{-m}|P|^{-r}|(W'W)^{-1/2}\bar{W}'P^{-1}W(W'W)^{-1/2}|^{-m}=\\
&=&|\bar{B}'B|^{-r}|W'W|^{-m}|P|^{-r}|W'W|^m|\bar{W}'P^{-1}W|^{-m}=\\
&=&|\bar{B}'B|^{-r}|P|^{-r}|\bar{H}_Y'(\bar{B}^{-1})'P^{-1}B^{-1}H_Y|^{-m}=\\
&=&|BP\bar{B}'|^{-r}|\bar{H}_Y'(BP\bar{B}')^{-1}H_Y|^{-m},\\
\end{eqnarray*}
which is the density of CMACG($BP\bar{B}'$).
\end{proof}

The features stated in Theorem \ref{th:gZZ} and Corollary \ref{cor:BZ} let us define a more general class of random matrices with orientations having CMACG($P$) distribution.

\begin{theorem}
Assume that an $m\times r$ random complex matrix $Z$ has the density of the form
\begin{equation}
f_Z(Z)=|P|^{-r}g(\bar{Z}'P^{-1}Z)
\end{equation}
invariant under right unitary transformation ($Z\rightarrow ZQ$, $\bar{Q}'Q=I_r$) with $P$ being an $m\times m$ positive define matrix, then its orientation $H_Z$ has the CMACG($P$) distribution.
\label{th:CMACGprop}
\end{theorem}

\begin{proof}
The proof is a straightforward generalization of the proof to the Theorem 3.2 in \citet{Chikuse1990}.\\
There exists a matrix $B$ such that $P=B\bar{B}'$ with $|B|\neq 0$. Define $U=B^{-1}Z$, so the distribution of $U$ is
$$f_U(U)=|P|^{-r}g(\bar{U}'U)|P|^r=g(\bar{U}'U),$$ 
which is also invariant under right unitary transformation ($U\rightarrow UQ$, $\bar{Q}'Q=I_r$).\\
According to Theorem \ref{th:gZZ} the orientation of $U$ is uniformly distributed on $V_{r,m}^{\mathbb{C}}$, i.e. $H_U\sim CMACG(I_m)$.\\
From Corollarly \ref{cor:BZ} applied to the orientation of the matrix $Z=BU$ we obtain that
$$H_Z\sim CMACG(BI_m\bar{B}')=CMACG(P).$$
\end{proof}

\section{Sampling from the CMACG distribution}

As pointed in the Introduction, one of the advantages of the CMACG distribution, especially for Bayesians, is an easy way for obtaining a pseudo-random sample from it. To get it the researcher may use exactly the definition of the considered distribution. Note, that \citet{Koop_al2009} us the same strategy in the real case. Suppose that one needs the sample from CMACG($P$) distribution with the known matrix parameter $P$. According to Theorem \ref{th:CMACGdef} the orientation $H_Z$ of $Z$ - the normally distributed complex random matrix (i.e. $Z\sim mN^{\mathbb{C}}(0,I_r,P)$) has CMACG($P$) distribution, so one draw is generated in two steps:
\begin{enumerate}
\item Generate an $m\times r$ matrix $Z$ from $mN^{\mathbb{C}}(0,I_r,P)$.
\item Put $H_Z=Z(\bar{Z}'Z)^{-\frac{1}{2}}$, where $(\bar{Z}'Z)^{-\frac{1}{2}}$ is the inverse of the square root of $\bar{Z}'Z$ .
\end{enumerate}
It was mentioned that due to its invariance property CAMCG($P$) may be treated as distribution definite for the elements of the complex Grassmann manifold, so by putting $P_Z=H_Z\bar{H}_Z'$ we obtain the projection matrix from the desired distribution.\\
The above points need additional comments. Firstly, the easiest way to obtain the draw from the complex matrix variate distribution is to employ its relationship with the real case, i.e. the condition that $Z=Z_R+iZ_I$, where $i=\sqrt{-1}$ has complex normal distribution  $Z\sim mN^{\mathbb{C}}(0,I_r,P)$, where $P=P_R+iP_I$ is a Hermitian matrix, is equivalent that its real and imaginary part are jointly normally distributed $\left(\begin{array}{c}Z_R\\Z_I\end{array}\right)|\sim mN\left(\left(\begin{array}{c}0\\0\end{array}\right),I_{r},\frac{1}{2}\left(\begin{array}{rr}P_R&-P_I\\P_I&P_R\end{array}\right)\right)$. Secondly, the square root of a complex Hermitan matrix, $(\bar{Z}'Z)^{\frac{1}{2}}$, needed in the polar decomposition of $Z$, can be obtained with the Newton's method proposed by \citet{Highman1986}.

\section{Conclusions}
\label{sec:conc}
This paper extends the matrix angular central distribution proposed by \citet{Chikuse1990} to the complex case. Considering the polar decomposition of a random complex matrix and the appropriate decomposition of measures we obtained the density function of this matrix's orientation, which is the element of the complex Stiefel manifold. We show that this new distribution inherits the properties after MACG distribution. We also discuss the way of obtaining a pseudo-random sample from this distribution.\\
It is worth emphasizing again that the complex matrix angular central distribution might be useful in the Bayesian analysis of VEC models with complex unit roots, e.g. in seasonally cointegrated VAR models.

\end{document}